\documentclass[12pt,a4paper]{article}
\usepackage[a4paper, tmargin=3.4cm, bmargin=3.4cm, lmargin=2cm, rmargin=2cm, textheight=24cm, textwidth=16cm]{geometry}
\usepackage{setspace}
\usepackage{amsfonts}
\usepackage{tikz}
\usepackage{epsfig}
\usepackage{latexsym}
\usepackage{amsthm}
\usepackage{dsfont}
\usepackage{hyperref}
\usepackage{xcolor}
\hypersetup{
	colorlinks,
	linkcolor={red!50!black},
	citecolor={blue!50!black},
	urlcolor={blue!80!black}
}
\usepackage{amsmath}
\usepackage{amssymb}

\usepackage[mathscr]{euscript}
 \let\mathscr\relax
\usepackage[scr]{rsfso}

\usepackage{array}
\usepackage{enumerate}
\usepackage{graphicx}
\usepackage{booktabs}
\usepackage{mathtools}
\usepackage{amsmath,blkarray,booktabs}

\newtheorem{theorem}{Theorem}[section]
\newtheorem{corollary}{Corollary}[section]

\newtheorem{lemma}{Lemma}[section]

\newtheorem{nota}{Notation}[section]
\newtheorem{example}{Example}[section]

\providecommand{\keywords}[1]
{
	\small	
	\textbf{Keywords:} #1
}
\begin{document}

	\title{A new generalization of Fielder's lemma with applications
	}
	
	\author{
		Komal Kumari and Pratima Panigrahi
		\\ \small Department of Mathematics, Indian Institute of Technology Kharagpur, India\\ \small e-mail: komalkumari1223w@kgpian.iitkgp.ac.in, pratima@maths.iitkgp.ac.in}
	
	
	\maketitle
	\begin{abstract}
		Very recently Ma and Wu \cite{wu2024generalization} obtained a generalization of Fielder's lemma and applied to find adjacency, Laplacian, and signless Laplacian spectra of $P_n-$ product of commuting graphs.  In this paper, we give a generalization of Fielder's  lemma applying which not only one gets generalized  result in  \cite{wu2024generalization} as a particular case, but also one can find several kind of spectra of $H$-product of graphs when $H$ is an arbitrary graph. Moreover, we  compute adjacency spectrum of $H-$ product  of commuting graphs and universal adjacency spectrum of $H-$ product of  commuting regular graphs.\\
	\end{abstract}
	
	\keywords{ Fiedler’s lemma; $H-$ product of graphs; commuting graphs; adjacency spectrum; universal adjacency spectrum.}\\
	\textbf{Mathematical Classification Code :}{ 05C50; 15A18}

	\section{Introduction}
	 We consider simple graph $G = (V(G),E(G))$ where $V(G)$ and $E(G)$  represent the sets of vertices and   edges of $G$ respectively. If $u$ and $v$ are end vertices of an edge, then we denote the edge by $uv$.  For any $u \in V(G)$,
	 the degree of $u$, denoted by
	$d_G(u)$, is  the number of neighbors of $u$ in $G$. If every vertex of $G$ has the same degree 
	 $r$, then $G$ is referred as an $r$-regular graph.\\
	 
	  Let $H$ be a graph with $V(H)= \{1,2,\ldots,l\}$, and let $H_1,H_2,\ldots,H_l$ be  graphs of the same order $n$ with $V(H_j)= \{v_{j1},v_{j2},\ldots,v_{jn}\}$, $j = 1,2,\ldots,l$. The $H-$product of $H_1,H_2,\ldots,H_l$, introduced by Howlader \cite{howlader2022distance}, denoted as $\widehat{G} = {\Lambda}_{H} (H_1,H_2, \ldots,H_l),$ is a graph with  $V(\widehat{G})= \bigcup_{j=1} ^{l}V(H_j) $ and $E(\widehat{G})= (\underset{j=1}{\overset{l}{\bigcup} }  E(H_j) )~ \bigcup~ (\underset{j,k \in E(H)}{\bigcup} \{v_{ji}v_{ki}: i = 1,2,\ldots,n\})$. For example, consider the graphs $H, H_1,H_2,H_3,$ and $H_4$ shown in  Figure \ref{fig:path5}, then the graph ${\Lambda}_{H} (H_1,H_2,H_3,H_4)$ can be represented in Figure \ref{fig:path6}.
	  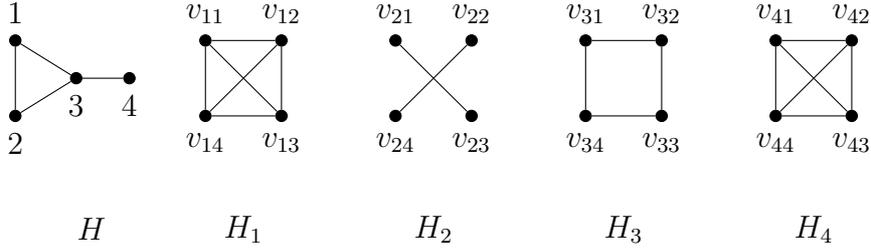
\begin{figure}[h]
	  	\centering
	  	\begin{tikzpicture}
	  		\node[draw, circle, fill=black, inner sep=1.5pt, label=above:$1$] (a1) at (0,0) {};
	  		\node[draw, circle, fill=black, inner sep=1.5pt, label=below:$2$] (a2) at (0,-1) {};
	  		\node[draw, circle, fill=black, inner sep=1.5pt, label=below:$3$] (a3) at (0.8,-0.5) {};
	  		\node[draw, circle, fill=black, inner sep=1.5pt, label=below:$4$] (a4) at (1.5,-0.5) {};
	  		\draw (a1) -- (a2);
	  		\draw (a2) -- (a3);
	  		\draw (a3) -- (a1);
	  		\draw (a4) -- (a3);
	  		\node at (1,-2.5) {$H$}; 

	  		\node[draw, circle, fill=black, inner sep=1.5pt, label=above:$v_{11}$] (v1) at (2.5,0) {};
	  		\node[draw, circle, fill=black, inner sep=1.5pt, label=above:$v_{12}$] (v2) at (3.5,0) {};
	  		\node[draw, circle, fill=black, inner sep=1.5pt, label=below:$v_{13}$] (v3) at (3.5,-1) {};
	  		\node[draw, circle, fill=black, inner sep=1.5pt, label=below:$v_{14}$] (v4) at (2.5,-1) {};
	  		
	  		\draw (v1) -- (v2);
	  		\draw (v2) -- (v3);
	  		\draw (v3) -- (v4);
	  		\draw (v4) -- (v1);
	  		\draw (v1) -- (v3);
	  		\draw (v2) -- (v4);
	  		\node at (3,-2.5) {$H_1$}; 
	  		\node[draw, circle, fill=black, inner sep=1.5pt, label=above:$v_{21}$] (u1) at (5,0) {};
	  		\node[draw, circle, fill=black, inner sep=1.5pt, label=above:$v_{22}$] (u2) at (6,0) {};
	  		\node[draw, circle, fill=black, inner sep=1.5pt, label=below:$v_{23}$] (u3) at (6,-1) {};
	  		\node[draw, circle, fill=black, inner sep=1.5pt, label=below:$v_{24}$] (u4) at (5,-1) {};
	  		
	  		\draw (u1) -- (u3);
	  		\draw (u2) -- (u4);
	  		\node at (5.5,-2.5) {$H_2$}; 
	  		
	  		\node[draw, circle, fill=black, inner sep=1.5pt, label=above:$v_{31}$] (t1) at (7.5,0) {};
	  		\node[draw, circle, fill=black, inner sep=1.5pt, label=above:$v_{32}$] (t2) at (8.5,0) {};
	  		\node[draw, circle, fill=black, inner sep=1.5pt, label=below:$v_{33}$] (t3) at (8.5,-1) {};
	  		\node[draw, circle, fill=black, inner sep=1.5pt, label=below:$v_{34}$] (t4) at (7.5,-1) {};
	  		
	  		\draw (t1) -- (t2);
	  		\draw (t2) -- (t3);
	  		\draw (t3) -- (t4);
	  		\draw (t4) -- (t1);
	  		\node at (8,-2.5) {$H_3$};

	  		\node[draw, circle, fill=black, inner sep=1.5pt, label=above:$v_{41}$] (w1) at (10,0) {};
	  		\node[draw, circle, fill=black, inner sep=1.5pt, label=above:$v_{42}$] (w2) at (11,0) {};
	  		\node[draw, circle, fill=black, inner sep=1.5pt, label=below:$v_{43}$] (w3) at (11,-1) {};
	  		\node[draw, circle, fill=black, inner sep=1.5pt, label=below:$v_{44}$] (w4) at (10,-1) {};
	  		
	  		\draw (w1) -- (w2);
	  		\draw (w2) -- (w3);
	  		\draw (w3) -- (w4);
	  		\draw (w4) -- (w1);
	  		\draw (w1) -- (w3);
	  		\draw (w2) -- (w4);
	  		\node at (10.5,-2.5) {$H_4$};

	  	\end{tikzpicture}
	  	\caption{ The graph  \(H, H_1, H_2, H_3 ,H_4\).}
	  	\label{fig:path5}
	  \end{figure} 
	  
	  \begin{figure}[h]
	  	\centering
	  	\begin{tikzpicture}		
	  		
	  		\node[draw, circle, fill=black, inner sep=1.5pt, label=above:$v_{11}$] (v1) at (0,0) {};
	  		\node[draw, circle, fill=black, inner sep=1.5pt, label=above:$v_{12}$] (v2) at (1,0) {};
	  		\node[draw, circle, fill=black, inner sep=1.5pt, label=below:$v_{13}$] (v3) at (1,-1) {};
	  		\node[draw, circle, fill=black, inner sep=1.5pt, label=below:$v_{14}$] (v4) at (0,-1) {};
	  		
	  		
	  		\node[draw, circle, fill=black, inner sep=1.5pt, label=above:$v_{21}$] (u1) at (0.8,-2.8) {};
	  		\node[draw, circle, fill=black, inner sep=1.5pt, label=above:$v_{22}$] (u2) at (1.8,-2.8) {};
	  		\node[draw, circle, fill=black, inner sep=1.5pt, label=below:$v_{23}$] (u3) at (1.8,-3.8) {};
	  		\node[draw, circle, fill=black, inner sep=1.5pt, label=below:$v_{24}$] (u4) at (0.8,-3.8) {};
	  		

	  		\node[draw, circle, fill=black, inner sep=1.5pt, label=above:$v_{31}$] (t1) at (3,-1.5) {};
	  		\node[draw, circle, fill=black, inner sep=1.5pt, label=above:$v_{32}$] (t2) at (4,-1.5) {};
	  		\node[draw, circle, fill=black, inner sep=1.5pt, label=below:$v_{33}$] (t3) at (4,-2.5) {};
	  		\node[draw, circle, fill=black, inner sep=1.5pt, label=below:$v_{34}$] (t4) at (3,-2.5) {};

	  		\node[draw, circle, fill=black, inner sep=1.5pt, label=above:$v_{41}$] (w1) at (6,0) {};
	  		\node[draw, circle, fill=black, inner sep=1.5pt, label=above:$v_{42}$] (w2) at (7,0) {};
	  		\node[draw, circle, fill=black, inner sep=1.5pt, label=below:$v_{43}$] (w3) at (7,-1) {};
	  		\node[draw, circle, fill=black, inner sep=1.5pt, label=below:$v_{44}$] (w4) at (6,-1) {};
	  		\draw (v1) -- (v2);
	  		\draw (v2) -- (v3);
	  		\draw (v3) -- (v4);
	  		\draw (v4) -- (v1);
	  		\draw (v1) -- (v3);
	  		\draw (v2) -- (v4);
	  		\draw (u1) -- (u3);
	  		\draw (u2) -- (u4);
	  		\draw (t1) -- (t2);
	  		\draw (t2) -- (t3);
	  		\draw (t3) -- (t4);
	  		\draw (t4) -- (t1);
	  		\draw (w1) -- (w2);
	  		\draw (w2) -- (w3);
	  		\draw (w3) -- (w4);
	  		\draw (w4) -- (w1);
	  		\draw (w1) -- (w3);
	  		\draw (w2) -- (w4);
	  		\draw (v1) -- (u1);
	  		\draw (v2) -- (u2);
	  		\draw (v3) -- (u3);
	  		\draw (v4) -- (u4);		
	  		\draw (t1) -- (u1);
	  		\draw (t2) -- (u2);
	  		\draw (t3) -- (u3);
	  		\draw (t4) -- (u4);
	  		\draw (t1) -- (w1);
	  		\draw (t2) -- (w2);
	  		\draw (t3) -- (w3);
	  		\draw (t4) -- (w4);
	  		\draw (t1) -- (v1);
	  		\draw (t2) -- (v2);
	  		\draw (t3) -- (v3);
	  		\draw (t4) -- (v4);
	  		
	  	\end{tikzpicture}
	  	\caption{ The graph  \({\Lambda}_{H} (H_1,H_2, \ldots,H_l)\)}
	  	\label{fig:path6}
	  \end{figure}
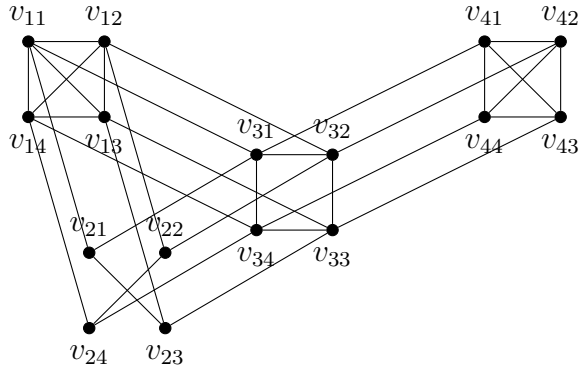
	   We recall that the  $H-$ join of graphs $H_1, H_2, \ldots, H_l$, is denoted by $H[H_1,H_2, \ldots, H_l]$, is constructed from $H_1 \cup H_2 \cup \ldots \cup H_l$ by connecting each vertex of $H_i$ to every vertex of $H_j$ whenever $ij \in E(H)$.
	  The concept of an  almost equitable $l$-partition of the vertices of a graph $G$, introduced by Cardoso et al. \cite{cardoso2007laplacian}, is a partition $ \Pi = (V_1,V_2,\ldots,V_l)$ of $V(H)$, where  every vertex in $V_i$ has exactly $d_{ij}$ neighbors in $V_{j}$ for each $j \neq i$.  \\
	 
	 The adjacency matrix of a graph $G$ with $n$ vertices, denoted by $A(G)$, is a $0-1$ matrix of size  $n \times n$, wherein the entry  $a_{ij} = 1$
	 if the  $i^{th}$ vertex is adjacent to the $j^{th}$ vertex, and  $0$ otherwise.  For  $k\geq2$, graphs $G_1,G_2,\ldots, G_k$ of the same order,  are called commuting graphs if their adjacency matrices commute pairwise. The diagonal matrix whose $i^{th}$ diagonal entry is the degree of $i^{th}$ vertex in $G$, is denoted by $ D(G)$. Additionally, $I$, $J$, and  $O$ represent the identity matrix, the all-ones matrix, and the zero matrix, respectively. For  $\alpha,\beta,\gamma,\eta \in \mathbb{R}$, the universal adjacency matrix of  graph 
	 $G$ is defined as 
	 $U(G)=\alpha A(G)+ \beta D(G)+ \gamma I+\eta J$. By substituting specific values for 
	$(\alpha,\beta,\gamma,\eta) = (1,0,0,0), (-1,1,0,0),(1,1,0,0)$, and $(-2,0,-1,1)$, one can obtain respectively the adjacency matrix $A(G)$, the Laplacian matrix $L(G)$, the signless Laplacian matrix $Q(G)$, and the Seidel matrix of the graph. For an 	$n \times n$ real symmetric matrix $M$ with distinct eigenvalues  $ \lambda_1,  \lambda_2,  \ldots, \lambda_{s}$   and their corresponding  multiplicities $k_1,k_2,\ldots,k_s$, the  spectrum of $M$ is represented as Spec$(M)$ 
	 $= \{ {\lambda_1}^{(k_1)}, {\lambda_2}^{(k_2)}, \ldots, {\lambda_s}^{(k_s)} \}$. If  $\lambda$ is an eigenvalue of $M$ with corresponding eigenvector  $\textbf{u}$, then the  pair $(\lambda, \textbf{u})$ is known as an eigenpair of $M$. If $(\lambda_i, \textbf{u}_i)$, $i = 1,2,\ldots,k$ are eigenpairs of $M$ such that $\{\textbf{u}_1, \textbf{u}_2, \ldots, \textbf{u}_k\}$ is an orthonormal set, then  $\{(\lambda_i, \textbf{u}_i): i = 1,2,\ldots,k \}$ is called a set of orthonormal eigenpairs. \\

      In 1974, Fielder presented the following lemma.
		\begin{lemma}\cite{fiedler1974eigenvalues} [  Fielder's Lemma]\label{lemma1.1}
		Let $A$ and $B$ be symmetric matrices of sizes $m$ and $n$, respectively, with  eigenpairs $(\alpha_i,u_i),$ $ i = 1,2,\ldots, m $ and $(\beta_i,v_i),$ $i = 1,2,\ldots,n$, respectively. Assume that $\|u_1\|= 1=\|v_1\|.$ Then, for any $\rho$, the matrix 
		\begin{equation*}
			C=	\begin{pmatrix}
				A & \rho u_1 v_1 ^T\\
				\rho v_1 u_1 ^T& B 
			\end{pmatrix}
		\end{equation*}
		has eigenvalues $\alpha_2,\alpha_3,\ldots,\alpha_m,\beta_2,\beta_3,\ldots,\beta_n,\gamma_1,\gamma_2$, where  $\gamma_1,\gamma_2$ are the eigenvalues of the matrix 
		\begin{equation*}
			\hat{C}= \begin{pmatrix}
				\alpha_1 & \rho\\
				\rho& \beta_1
			\end{pmatrix}.
		\end{equation*}
	\end{lemma}
 Robbiano et al. \cite{robbiano2010extending} extended Fielder's lemma and utilized it to determine the eigenspaces of specific graphs.
 Cardoso et al. \cite{cardoso2011generalization} extended Fielder's lemma where they considered more than two block diagonal square matrices and then applied the result to determine the eigenvalues of $P_n-$ join of regular graphs, where $P_n$ is the path on $n$ vertices. For an arbitrary graph $H$, Cardoso et al. \cite{cardoso2013spectra} derived the adjacency spectrum for $H-$join of  regular graphs, as well as the Laplacian spectra for $H-$join  of arbitrary graphs.  Wu et al. \cite{wu2014signless} provided the signless Laplacian and normalized Laplacian spectra for $H-$join of regular graphs.  Andrade et al. \cite{andrade2017spectra} applied Fielder's lemma to examine the Randić spectra of $H-$join of regular graphs.
Most  recently, Ma and Wu \cite{wu2024generalization} presented a result, as given in Theorem \ref{theorem1} below, which is a  generalization of the generalized result of Fielder's lemma in Robbiano et al. \cite{robbiano2010extending}. Then the same authors applied their result and  determined adjacency, Laplacian, and signless Laplacian spectra of $P_n$- product of commuting graphs. 

\begin{theorem}\label{theorem1}\cite{wu2024generalization}
	For $j = 1,2,\ldots,n,$ let $A_j$ be $m_j \times m_j$ symmetric matrices, with corresponding eigenpairs $(\alpha_{i,j},u_{i,j}),$ $i = 1,2,\ldots,m_j$.  Let $ k \leq  min\{m_1,m_2,\ldots,m_n\},$ and $U_j = (u_{1,j}|u_{2,j}|\ldots|u_{k,j})$, $j= 1,2,\ldots,n.$ Furthermore, suppose that for each $j$, the system of eigenvectors $u_{i,j}$, $i = 1,2,\ldots,m_j,$  is orthonormal. Then, for $\rho_1,\rho_2,\ldots,\rho_{n-1},$ the matrix
	
	\begin{equation*}
		C= \begin{pmatrix}
			A_1 & \rho_{1} U_1 U_2 ^T&  &  &\\
			 \rho_{1} U_2 U_1 ^T & A_2&\rho_{2}U_2U_3 ^T & & \\
			   &\ddots&\ddots&\ddots& \\
			   & & \rho_{n-2}U_{n-1}U_{n-2} ^T&A_{n-1}&\rho_{n-1}U_{n-1} U_{n} ^T\\
			    & & & \rho_{n-1}U_n U_{n-1}&A_n		
		\end{pmatrix}
	\end{equation*}
has eigenvalues $ \lambda_{k+1,1}, \lambda_{k+2,1}, \ldots,\lambda_{m_1,1},\lambda_{k+1,2},\lambda_{k+1,2},\ldots,\lambda_{m_2,2},\ldots,\lambda_{k+1,n},\lambda_{k+2,n},\ldots,\lambda_{m_n,n},$ $\xi_{1,t},$ $\xi_{2,t},$ $\ldots,\xi_{n,t},$ where $\xi_{1,t},\xi_{2,t},\ldots,\xi_{n,t}$, $t = 1,2,\ldots,k$ are eigenvalues of the matrix
\begin{equation*}
\hat{C}_t = 	\begin{pmatrix}
          \lambda_{t,1} & \rho_{1} &  &  &\\
             \rho_{1} & \lambda_{t,2}&\rho_{2} & & \\
                       &\ddots&\ddots&\ddots& \\
              & & \rho_{n-2}&\lambda_{t,n-1}&\rho_{n-2}\\
           &  & & \rho_{n-1}&\lambda_{t,n}
       \end{pmatrix}.      
  \end{equation*}
\end{theorem}

In this paper, we generalize Theorem \ref{theorem1} stated above. Then applying our generalized result, we find adjacency spectrum of $H$-product of commuting graphs and  universal adjacency spectrum of commuting regular graphs, for any arbitrary simple graph $H$. It is to be noted that by applying result of Ma and Wu \cite{wu2024generalization}, one cannot find Seidel spectrum of, even,   $P_n$-product of graphs. 

\section{Generalization of Theorem \ref{theorem1}}
 
 In this section, we further extend  Fielder's lemma and Theorem \ref{theorem1}  as given below. 

\begin{theorem}\label{theorem2}
	 For $j = 1,2,\ldots,n,$ let $A_j$ be real symmetric matrices of size $m_j \times m_j$, with corresponding eigenpairs $(\lambda_{i,j},u_{i,j})$, $i = 1,2,\ldots,m_j$. Let $ k \leq  min\{m_1,m_2,\ldots,m_n\},$ and $U_j = (u_{1,j}|u_{2,j}|\ldots|u_{k,j})$, $j= 1,2,\ldots,n.$ Furthermore, suppose that for each $j$, the set of eigenvectors $u_{i,j}$, $i = 1,2,\ldots,m_j,$  is orthonormal.  Let $\rho_{i,j}$, $1 \leq i \neq j \leq k,$ be  real numbers such that $\rho_{i,j} = \rho_{j,i}$.  Then, the eigenvalues of the matrix 
\begin{equation*}
	C= \begin{pmatrix}
		A_1 & \rho_{1,2} U_1 U_2 ^T&\ldots&\rho_{1,n}U_1U_n ^T\\
		\rho_{2,1} U_2 U_1 ^T & A_2&\ldots&\rho_{2,n}U_2U_n ^T\\
		\vdots&\vdots&\ddots&\vdots\\
		\rho_{n,1}U_nU_1 ^T&\rho_{n,2}U_nU_2 ^T&\ldots&A_n
	\end{pmatrix}
\end{equation*}
  are $ \lambda_{k+1,1}, \lambda_{k+2,1}, \ldots,\lambda_{m_1,1},\lambda_{k+1,2},\lambda_{k+1,2},\ldots,\lambda_{m_2,2},\ldots,$ $\lambda_{k+1,n}, \lambda_{k+2,n},\ldots,\lambda_{m_n,n},\xi_{1,t},\xi_{2,t},\ldots,\xi_{n,t},$ where $\xi_{1,t},\xi_{2,t},\ldots,\xi_{n,t}$, $t = 1,2,\ldots,k,$ are the eigenvalues of the matrix 
\begin{equation*}
	{C}_t= \begin{pmatrix}
		\lambda_{t,1} & \rho_{1,2}  &\ldots&\rho_{1,n}\\
		\rho_{2,1}  & \lambda_{t,2}&\ldots&\rho_{2,n}\\
		\vdots&\vdots&\ddots&\vdots\\
		\rho_{n,1}&\rho_{n,2}&\ldots&\lambda_{t,n}
	\end{pmatrix}.
\end{equation*}
\end{theorem}
\begin{proof}
For $s = 1,2,\ldots,n$, $t = 1,2,\ldots,k,$ let $(\xi_{s,t},\widehat{w_{s,t}})$ be eigenpairs of the matrix ${C}_t$, where $\widehat{w_{s,t}} = (w_{1st},w_{2st}, \ldots,w_{nst})^T.$ Then, we have 
\begin{equation*}
	{C}_t \widehat{{w}_{s,t}}= \begin{pmatrix}
	\lambda_{t,1} & \rho_{1,2} &\ldots&\rho_{1,n}\\
		\rho_{2,1}  &\lambda_{t,2}&\ldots&\rho_{2,n}\\
		\vdots&\vdots&\ddots&\vdots\\
		\rho_{n,1}&\rho_{n,2}&\ldots&\lambda_{t,n}
	\end{pmatrix} 
\begin{pmatrix}
	w_{1st}\\
	w_{2st}\\
	\vdots\\
	w_{nst}
\end{pmatrix}
\end{equation*}	

\begin{equation}\label{eq1}
=	\begin{pmatrix}
		w_{1st} \lambda_{t,1} + w_{2st} \rho_{1,2}+ \ldots+ w_{nst}\rho_{1,n}\\
		w_{1st}\rho_{2,1}+ w_{2st}\lambda_{t,2}+ \ldots+ w_{nst} \rho_{2,n}\\
		\vdots\\
			w_{1st}\rho_{n,1}+ w_{2st}\rho_{n,2}+ \ldots+  w_{nst}\lambda_{t,n}\\
			\end{pmatrix}
		=  \begin{pmatrix}
		\xi_{s,t}	w_{1st}\\
		\xi_{s,t}	w_{2st}\\
			\vdots\\
		\xi_{s,t}	w_{nst}
		\end{pmatrix}
	= \xi_{s,t} \begin{pmatrix}
		w_{1st}\\
		w_{2st}\\
		\vdots\\
		w_{nst}
	\end{pmatrix}
\end{equation}
Let $\begin{pmatrix}
		w_{1st} u_{t,1}\\
    	w_{2st} u_{t,2}\\
    	\vdots\\
    	w_{nst}u_{t,n}
\end{pmatrix}$ be a vector  of dimension $(m_1+m_2+\ldots+m_n)$. We show that this is an eigenvector of the matrix $C$. 
\begin{equation*}
\begin{pmatrix}
	A_1 & \rho_{1,2} U_1 U_2 ^T&\ldots&\rho_{1,n}U_1U_n ^T\\
	\rho_{2,1} U_2 U_1 ^T & A_2&\ldots&\rho_{2,n}U_2U_n ^T\\
	\vdots&\vdots&\ddots&\vdots\\
	\rho_{n,1}U_nU_1 ^T&\rho_{n2}U_nU_2 ^T&\ldots&A_n
\end{pmatrix}
\begin{pmatrix}
	w_{1st} u_{t,1}\\
	w_{2st} u_{t,2}\\
	\vdots\\
	w_{nst}u_{t,n}
\end{pmatrix} 
\end{equation*}

\begin{equation*}
=\begin{pmatrix}
	w_{1st}A_1 u_{t,1}+ w_{2st} \rho_{1,2} U_1U_2 ^Tu_{t,2}+ \ldots +  w_{nst} \rho_{1,n} U_1U_n ^Tu_{t,n}\\
	 w_{1st} \rho_{2,1} U_2U_1 ^Tu_{t,1}+	w_{2st}A_2 u_{t,2}+ \ldots +  w_{nst} \rho_{n,2} U_2U_n ^Tu_{t,n}\\
	 \vdots\\
	  w_{1st} \rho_{n,1} U_nU_1 ^Tu_{t,1}+	w_{2st} \rho_{n,2} U_nU_2 ^Tu_{t,2}+ \ldots +  w_{nst} A_n u_{t,n}\\
	\end{pmatrix}
\end{equation*}

Based on the definition  of $U_j = (u_{1,j}|u_{2,j}|\ldots|u_{k,j}),~j= 1,2,\ldots,n,$ the above matrix is 

\begin{equation*}
\begin{pmatrix}
	w_{1st}\lambda_{t,1} u_{t,1}+ w_{2st} \rho_{1,2} U_1e_t+ \ldots +  w_{nst} \rho_{1,n} U_1e_t\\
	w_{1st} \rho_{2,1} U_2e_t+	w_{2st} \lambda_{t,2} u_{t,2}+ \ldots +  w_{nst} \rho_{n,2} U_2e_t\\
	\vdots\\
	w_{1st} \rho_{n,1} U_ne_t+	w_{2st} \rho_{n,2} U_ne_t+ \ldots +  w_{nst} \lambda_{t,n} u_{t,n}\\
\end{pmatrix}
=
\begin{pmatrix}
	w_{1st}\lambda_{t,1} u_{t,1}+ w_{2st} \rho_{1,2} u_{t,1}+ \ldots +  w_{nst} \rho_{1,n} u_{t,1}\\
	w_{1st} \rho_{2,1}  u_{t,2}+	w_{2st} \lambda_{t,2} u_{t,2}+ \ldots +  w_{nst} \rho_{n,2} u_{t,2}\\
	\vdots\\
	w_{1st} \rho_{n,1} u_{t,n}+	w_{2st} \rho_{n,2}  u_{t,n}+ \ldots +  w_{nst} \lambda_{t,n} u_{t,n}\\
\end{pmatrix}
\end{equation*}

Then,  by Equation (\ref{eq1}), this column matrix is\\
 \begin{equation*}
\begin{pmatrix}
	(w_{1st}\lambda_{t,1}+ w_{2st} \rho_{1,2} + \ldots +  w_{nst} \rho_{1,n}) u_{t,1}\\
	(w_{1st} \rho_{2,1} +	w_{2st} \lambda_{t,2} + \ldots +  w_{nst} \rho_{n,2}) u_{t,2}\\
	\vdots\\
	(w_{1st} \rho_{n,1} +	w_{2st} \rho_{n,2}  + \ldots +  w_{nst} \lambda_{t,n}) u_{t,n}\\
\end{pmatrix}
= \begin{pmatrix}
	\xi_{s,t}	w_{1st}u_{t,1} \\
	\xi_{s,t}	w_{2st} u_{t,2}\\
	\vdots\\
	\xi_{s,t}	w_{nst}u_{t,n}
\end{pmatrix}
=  	\xi_{s,t} \begin{pmatrix}
		w_{1st}u_{t,1} \\
		w_{2st} u_{t,2}\\
	\vdots\\
	w_{nst}u_{t,n}
\end{pmatrix}
\end{equation*}

Thus, for $ s= 1,2,\ldots, n$ and $ t= 1,2, \ldots, k,$ the pairs $\bigg( \xi_{s,t}, \begin{pmatrix}
	w_{1st}u_{t,1}\\
	w_{2st}u_{t,2}\\
	\ldots \\
	w_{nst} u_{t,n}
\end{pmatrix} \bigg)$ represent $nk$ eigenpairs of the matrix $C$. Next, for $j = 1,2,\ldots,n,$ and the pairs $(\lambda_{i,j},u_{i,j})$ with  $i = k+1,k+2,\ldots,m_j,$ the following : 
\begin{equation*}
	\begin{pmatrix}
		A_1 & \rho_{1,2} U_1 U_2 ^T\ldots&\rho_{1,n}U_1U_n ^T\\
		\rho_{2,1} U_2 U_1 ^T & A_2\ldots&\rho_{2,n}U_2U_n ^T\\
		\vdots&\vdots&\ddots&\vdots\\
		\rho_{n,1}U_nU_1 ^T&\rho_{n,2}U_nU_2 ^T&\ldots&A_n
	\end{pmatrix}
\begin{pmatrix}
	0\\
	\vdots\\
	u_{i,j}\\
	\vdots\\
	0
\end{pmatrix}
= \begin{pmatrix}
	0\\
	\vdots\\
	A_j u_{i,j}\\
	\vdots\\
	0
\end{pmatrix}
= \alpha_{i,j}\begin{pmatrix}
	0\\
	\vdots\\
	u_{i,j}\\
	\vdots\\
	0
\end{pmatrix}
\end{equation*}

Therefore, for $j=1,2,\ldots,n$, $i = k+1,k+2,\ldots,m_j$, the pairs $(\lambda_{i,j},u_{i,j})$ are eigenpairs of $C$. Hence the result follows.
\end{proof}

\section{Spectra of $H$-product of commuting  graphs}
In this section, we apply Theorem \ref{theorem2}   and obtain spectra of $H$-product of commuting graphs.  Akbari et al. \cite{akbari2007commuting, akbari2009commutativity} discussed several properties and provided examples of commuting graphs. The lemma below provides a characterization of commuting matrices in terms of their eigenvectors.
\begin{lemma}\label{lemma1} \cite{heinze2001applications}
	Let $A_1,A_2,\ldots,A_n$ be symmetric matrices of order $k$. Then the following are equivalent.
	\begin{enumerate}
		\item $A_i A_j =A_j A_i$,  for any $i,j \in \{1,2,\ldots n\}$.
		\item There exists an orthonormal basis $\{x_1,x_2,\ldots,x_k\}$ of $\mathbb{R}^k$ such that $x_1,x_2,\ldots,x_k$ are eigenvectors of $A_i,$ $i = 1,2,\ldots,n.$
	\end{enumerate} 
	
\end{lemma}

\begin{nota}\rm{
In the rest of the paper, we consider an arbitrary simple graph $H$ with $V(H)=1,2,\ldots,l$. Also we consider vertex disjoint simple graphs $H_1.H_2,\ldots,H_l,$ of order $n$ each.
For integers $n$ and $r$, we denote by $e_{n,r}$ a column vector of dimension $n$ where  the $r^{th}$ entry is $1$ and all others are  $0$. The all-ones column vector of dimension $n$ is denoted by $j_ n$.}
\end{nota}

\subsection{Adjacency Spectrum of  $ {\Lambda}_{H} (H_1,H_2, \ldots,H_l)$ }

Here we consider the $H-$product of graphs  $\widehat{G} = {\Lambda}_{H} (H_1,H_2, \ldots,H_l)$ as defined in the introduction. We note that the graph   $\widehat{G} = {\Lambda}_{H} (H_1,H_2, \ldots,H_l)$ has an almost equitable  $l$ partition $\Pi(\widehat{G})= \{V(H_1),V(H_2),\ldots,V(H_l)\}$ such that 
\begin{equation*}
	d_{jk}  = \begin{cases}
		1 \text{~if~} jk \in E(H)\\
		0 \text{~otherwise}, 
	\end{cases}
	\text{~for~} 1 \leq j,k \leq l.
\end{equation*} 
\begin{theorem}\label{theorem3.1}
	Suppose  $H_1,H_2,\ldots,H_l$ are  commuting graphs. Let $\lambda_{i,j}$ represent the eigenvalues of $H_j,$ where $i = 1,2, \ldots,n,$ $ j = 1,2,\ldots,l$, and let $u_1,u_2,\ldots,u_n$ form an orthonormal set of eigenvectors for the adjacency matrix $A(H_1)$. Then, the adjacency spectrum  of    $ {\Lambda}_{H} (H_1,H_2, \ldots,H_l)$  consists of   the eigenvalues of 
	\begin{equation}
		{C_i}=	\begin{pmatrix}
			\lambda_{i,1} & \rho_{1,2}&\ldots&\rho_{1,(l-1)}&\rho_{1,l}\\
			\rho_{2,1}&\lambda_{i,2}& \ldots&\rho_{2,(l-1)}&\rho_{2,l}\\
			\vdots&\vdots&\ddots&\vdots&\vdots\\
			\rho_{(l-1),1}&\rho_{(l-1),2}&\ldots&\lambda_{i,(l-1)}&\rho_{(l-1),l}\\
			\rho_{l,1}&\rho_{l,2}&\ldots&\rho_{l,(l-1)}&\lambda_{i,l}
		\end{pmatrix}_{l \times l}
	\end{equation}
 where  $i = 1,2,\ldots,n$, and
\begin{equation*}
	\rho_{j,k} = \rho_{k,j} = \begin{cases}
		1 \text{~if~} jk \in E(H)\\
		0 \text{~otherwise}, 
	\end{cases}
\text{~for~} 1 \leq j,k \leq l.
\end{equation*} 
\end{theorem}

\begin{proof}
	By the definition of   $\widehat{G} = {\Lambda}_{H} (H_1,H_2, \ldots,H_l)$, the vertex set of $\widehat{G}$ is partitioned as $V(H_1) \cup V(H_2) \cup \ldots \cup V(H_l).$ Then, the adjacency matrix of  $\widehat{G} = {\Lambda}_{H} (H_1,H_2, \ldots,H_l)$ is 
		\begin{equation*} 
		\begin{pmatrix}
			A(H_1) & \rho_{1,2}I_n& \ldots& \rho_{1,(l-1)}I_n&\rho_{1,l}I_n\\
			\rho_{2,1}I_n& A(H_2)& \ldots& \rho_{2,(l-1)}I_n&\rho_{2,l}I_n\\
			\vdots& \vdots& \ddots&\vdots& \vdots\\
			\rho_{(l-1),1}I_n& \rho_{(l-1),2}I_n& \ldots&A(H_{l-1})&\rho_{(l-1),l}I_n\\
			\rho_{l,1}I_n&\rho_{l,2}I_n& \ldots& \rho_{,l(l-1)}I_n&A(H_l)
		\end{pmatrix} 
	\end{equation*}
Given that $H_1,H_2,\ldots,H_l$ are $l$ commuting graphs, and $S =\{u_1,u_2,\ldots,u_n\}$  an orthonormal set of eigenvectors of  $A(H_1),$  it follows from  Lemma \ref{lemma1} that $S$ also an orthonormal set of eigenvectors for $A(H_2), A(H_3),\ldots, A(H_{l-1}),$ and $A(H_l).$ This implies that $U =(u_1|u_2|\ldots|u_n)$ is an orthonormal matrix. By the spectral decomposition of a symmetric matrix, we have $UU^T =I_n.$ Therefore, by substituting $UU ^T$ for $I_n$ in the aforementioned matrix. We get 
	\begin{equation*}
	\begin{pmatrix}
		A(H_1) & \rho_{1,2}UU^T& \ldots& \rho_{1,(n-1)}UU^T&\rho_{1,n}UU^T\\
		\rho_{2,1}UU^T& A(H_2)& \ldots& \rho_{2,(n-1)}UU^T&\rho_{2,n}UU^T\\
		\vdots& \vdots& \ddots&\vdots& \vdots\\
		\rho_{(n-1),1}UU^T& \rho_{(n-1),2}UU^T& \ldots&A(H_{n-1})&\rho_{(n-1),n}UU^T\\
		\rho_{n,1}UU^T&\rho_{n,2}UU^T& \ldots& \rho_{n,(n-1)}UU^T&A(H_n)
	\end{pmatrix}. 
\end{equation*}
Then, the result follows directly from Theorem \ref{theorem2} and Lemma  \ref{lemma1}.
\end{proof} 


\subsection{Universal adjacency spectrum of $ {\Lambda}_{H} (H_1,H_2, \ldots,H_l)$ }

Let $H$ be a graph with $V(H)= \{1,2,\ldots,l\}$, and $H_1,H_2,\ldots,H_l$ graphs with $V(H_j)= \{v_{j1},v_{j2},\ldots,v_{jn}\}$ for each $j = 1,2,\ldots,l$. Then, the degree of the vertex $v_{ji}$ in  $\widehat{G} = {\Lambda}_{H} (H_1,H_2, \ldots,H_l)$ is given by
\begin{equation*}
	deg_{\widehat{G}} (v_{ji})= deg_{H_j}(v_{ji}) + \sum_{\underset{j \neq k}{j=1}} ^l \rho_{k,j} 
\end{equation*}

\begin{equation*}
		\rho_{j,k} = \rho_{k,j} = \begin{cases}
		1 \text{~if~} jk \in E(H)\\
		0 \text{~otherwise}, 
	\end{cases}
	\text{~for~} 1 \leq j,k \leq l.
\end{equation*}

For $\alpha \neq 0$ and $\beta,\gamma,\eta \in \mathbb{R}$, let $U(\widehat{G}) = \alpha A(\widehat{G})+ \beta D(\widehat{G})+ \gamma I+ \eta J$ be the universal adjacency matrix of  $\widehat{G} = {\Lambda}_{H} (H_1,H_2, \ldots,H_l)$. Then, $U(\widehat{G})$ can be expressed in terms of the universal adjacency matrices of $H_j$ as follows: 
\begin{equation*} 
U(\widehat{G})=	\begin{pmatrix}
		U(H_1)+ (\beta \sum_{\underset{j \neq 1}{j=1}} ^l \rho_{1,j} ) I & \rho_{1,2} \alpha I + \eta J & \ldots  & \rho_{1,l} \alpha I + \eta J\\
		\rho_{2,1} \alpha I + \eta J&	U(H_2)+ (\beta  \sum_{\underset{j \neq 2}{j=1}} ^l \rho_{2,j} )I& \ldots&  \rho_{2,l} \alpha I+ \eta J\\
		\vdots&\vdots& \ddots&\vdots\\
		\rho_{l,1} \alpha I + \eta J&  \rho_{l,2} \alpha I + \eta J& \ldots& U(H_l)+   (\beta {\sum_{\underset{j \neq l}{j=1}} ^l} \rho_{l,j})I\\
	\end{pmatrix}
\end{equation*}
\begin{theorem}\label{theorem4.1}
 	Suppose  $H_1,H_2,\ldots,H_l$ are  regular commuting graphs  such that  $H_j$ is an $r_j$ regular graph for $j = 1,2,\ldots,l$. Let $\lambda_{i,j}$ be the adjacency eigenvalues of $H_j,$ where $i = 1,2, \ldots,n,$ and  let $u_1,u_2,\ldots,u_n$ form an orthonormal set of eigenvectors for the universal adjacency matrix $U(H_1)$. Then, the universal adjacency spectrum of   $\widehat{G}= {\Lambda}_{H} (H_1,H_2, \ldots,H_l)$ consists of the spectrum of 
	\begin{equation*}
	{C_1}=	\scriptsize
	\begin{pmatrix}
		\alpha r_1+ (r_1+\sum_{\underset{j \neq 1}{j=1}} ^l \rho_{1,j}) \beta+ \gamma+n \eta & \rho_{1,2} \alpha+ n\eta& \ldots&   \rho_{1,l} \alpha+ n\eta\\
		\rho_{2,1} \alpha +n \eta& 	\alpha r_2+ (r_2+ \sum_{\underset{j \neq 2}{j=1}} ^l \rho_{2,j}) \beta+ \gamma+n \eta & \ldots&  \rho_{2,l} \alpha+ n\eta\\
		\vdots& \vdots& \ddots& \vdots\\
		
		\rho_{l,1}\alpha + n \eta& 	\rho_{l,2}\alpha + n \eta& \ldots & 	 \alpha r_l+ (r_l+{\sum_{\underset{j \neq l}{j=1}} ^l} \rho_{l,j}) \beta+ \gamma+n \eta \\
		\end{pmatrix}
	\end{equation*}
 and
 \begin{equation*} 
 		{C_i}=
 		\scriptsize
 \begin{pmatrix}
 	\lambda_{i,1}\alpha+  (r_1+{\sum_{j=2} ^l} \rho_{1,j}) \beta+ \gamma & \rho_{1,2} \alpha& \ldots&   \rho_{1,l}\alpha\\
 	\rho_{2,1}\alpha&  	\lambda_{i,2}\alpha+  (r_2+ {\sum_{\underset{j \neq 2}{j=1}} ^l} \rho_{2,j}) \beta+ \gamma & \ldots &  \rho_{2,l}\alpha\\
 	\vdots&\vdots&\ddots&\vdots\\
 	\rho_{l,1}\alpha & \rho_{l,2}& \ldots& 	\lambda_{i,l}\alpha+  (r_l+ {\sum_{\underset{j \neq l}{j=1}} ^l} \rho_{l,j}) \beta+ \gamma
 \end{pmatrix}
\end{equation*}
 for $ i = 2, 3,\ldots,n,$ where 
 \begin{equation*}
 	\rho_{j,k} = \rho_{k,j} = \begin{cases}
 		1 \text{~if~} jk \in E(H)\\
 		0 \text{~otherwise}, 
 	\end{cases}
 	\text{~for~} 1 \leq j,k \leq l.
 \end{equation*} 
\end{theorem}
\begin{proof}
	Let $\mu_{s,1}$ be the eigenvalue of $ {C_1}$   corresponding to the eigenvector $ \widehat{w_{s,1}} = (w_{1s1},w_{2s1},\ldots,w_{ls1})$ for $s= 1,2,\ldots,l$. Then,
	\begin{equation*}
	{C_1} \widehat{w_{s,1}}=	\scriptsize
	\begin{pmatrix}
	\alpha r_1+ (r_1+\sum_{\underset{j \neq 1}{j=1}} ^l \rho_{1,j}) \beta+ \gamma+n \eta & \rho_{1,2} \alpha+ n\eta& \ldots&   \rho_{1,l} \alpha+ n\eta\\
	\rho_{2,1} \alpha +n \eta& 	\alpha r_2+ (r_2+ \sum_{\underset{j \neq 2}{j=1}} ^l \rho_{2,j}) \beta+ \gamma+n \eta & \ldots&  \rho_{2,l} \alpha+ n\eta\\
	\vdots& \vdots& \ddots& \vdots\\
	
	\rho_{l1}\alpha + n \eta& 	\rho_{l,2}\alpha + n \eta& \ldots & 	 \alpha r_l+ (r_l+{\sum_{\underset{j \neq l}{j=1}} ^l} \rho_{l,j}) \beta+ \gamma+n \eta \\
\end{pmatrix}
\begin{pmatrix}
	w_{1s1}\\
	w_{2s1}\\
	\vdots\\
	w_{ls1}
\end{pmatrix}
\end{equation*}
\begin{equation}\label{eq4*}
=	\begin{pmatrix}
	(\alpha r_1+ (r_1+\sum_{\underset{j \neq 1}{j=1}} ^l \rho_{1,j}) \beta+ \gamma+n \eta)w_{1s1} + (\rho_{1,2} \alpha+ n\eta)w_{2s1}+ \ldots+ (\rho_{1,l} \alpha+ n\eta)w_{ls1}\\
	(\rho_{2,1} \alpha +n \eta)w_{1s1}+ 	(\alpha r_2+ (r_2+\sum_{\underset{j \neq 2}{j=1}} ^l \rho_{2,j}) \beta+ \gamma+n \eta)w_{2s1}+\ldots+ ( \rho_{2,l} \alpha+ n\eta)w_{ls1}\\
	\vdots\\
	(\rho_{l,1}\alpha + n \eta)w_{1s1} + 	(\rho_{l,2}\alpha + n \eta)w_{2s1}+ \ldots + (\alpha r_l+(r_l+{\sum_{\underset{j \neq l}{j=1}} ^l} \rho_{l,j}) \beta+ \gamma+n \eta)w_{ls1}
	\end{pmatrix} = \mu_{s,1} \begin{pmatrix}
	w_{1s1}\\
	w_{2s1}\\
	\vdots\\
	w_{ls1}
\end{pmatrix}
\end{equation}
Now, let $u_1 = \frac{1}{\sqrt{n}} j_n$ and $\frac{1}{\sqrt{n}}$ $\begin{pmatrix}
		w_{1s1}j_n\\
	w_{2s1}j_n\\
	\vdots\\
	w_{ls1}j_n
	\end{pmatrix}$
be a vector of dimension $nl$. We show that this is an eigenvector of the matrix $U(\widehat{G})$ corresponding to the eigenvalue $\mu_{s,1}$.
\begin{equation*} 
\begin{pmatrix}
	U(H_1)+ (\beta \sum_{\underset{j \neq 1}{j=1}} ^l \rho_{1,j} ) I & \rho_{1,2} \alpha I + \eta J & \ldots  & \rho_{1,l} \alpha I + \eta J\\
	\rho_{2,1} \alpha I + \eta J&	U(H_2)+ (\beta  \sum_{\underset{j \neq 2}{j=1}} ^l \rho_{2,j} )I& \ldots&  \rho_{2,l} \alpha I+ \eta J\\
	\vdots&\vdots& \ddots&\vdots\\
	\rho_{l,1} \alpha I + \eta J&  \rho_{l,2} \alpha I + \eta J& \ldots& U(H_l)+   (\beta {\sum_{\underset{j \neq l}{j=1}} ^l} \rho_{l,j})I\\
\end{pmatrix}
\begin{pmatrix}
	\frac{w_{1s1}}{\sqrt{n}}j_n\\
	\frac{w_{2s1}}{\sqrt{n}}j_n\\
	\vdots\\
	\frac{w_{ls1}}{\sqrt{n}}j_n
\end{pmatrix}
\end{equation*}
\begin{equation*}
=\frac{1}{\sqrt{n}}	\begin{pmatrix}
		w_{1s1}(U(H_1)+ (\beta\sum_{\underset{j \neq 1}{j=1}} ^l \rho_{1,j}) I)j_n+  w_{2s1}(\rho_{1,2} \alpha I + \eta J )j_n+ \ldots+  w_{ls1}(\rho_{1,l} \alpha I + \eta J)j_n\\
		w_{1s1}( \rho_{2,1} \alpha I + \eta J )j_n+w_{2s1}(U(H_2)+ (\beta \sum_{\underset{j \neq 2}{j=1}} ^l \rho_{2,j})I)j_n+ \ldots + 	w_{ls1}( \rho_{2,1} \alpha I + \eta J )j_n\\
		\vdots\\
		w_{1s1}( \rho_{l,1} \alpha I + \eta J )j_n+ 	w_{2s1}( \rho_{l,2} \alpha I + \eta J )j_n+\ldots+ w_{ls1}(U(H_l)+  (\beta{\sum_{\underset{j \neq l}{j=1}} ^l} \rho_{l,j})I)j_n\\
		\end{pmatrix}
\end{equation*}
\begin{equation*}
	=\frac{1}{\sqrt{n}}	\begin{pmatrix}
		w_{1s1}(U(H_1)j_n+ (\beta\sum_{\underset{j \neq 1}{j=1}} ^l \rho_{1,j}) Ij_n)+  w_{2s1}(\rho_{1,2} \alpha Ij_n + \eta J j_n )+ \ldots+  w_{ls1}(\rho_{1,l} \alpha Ij_n + \eta Jj_n)\\
		w_{1s1}( \rho_{2,1} \alpha I j_n + \eta J j_n)+w_{2s1}(U(H_2)j_n+ (\beta \sum_{\underset{j \neq 2}{j=1}} ^l \rho_{2,j})Ij_n)+ \ldots + 	w_{ls1}( \rho_{2,1} \alpha Ij_n + \eta Jj_n )\\
		\vdots\\
		w_{1s1}( \rho_{l,1} \alpha I j_n + \eta J j_n )+ 	w_{2s1}( \rho_{l,2} \alpha I j_n + \eta Jj_n )+\ldots+ w_{ls1}(U(H_l)j_n+  (\beta{\sum_{\underset{j \neq l}{j=1}} ^l} \rho_{l,j})Ij_n)\\
	\end{pmatrix}
\end{equation*}
 
 \begin{equation*}
 	=\frac{1}{\sqrt{n}} \begin{pmatrix}
 		w_{1s1}((r_1 \alpha + r_1 \beta + \gamma + n \eta)j_n +(\beta\sum_{\underset{j \neq 1}{j=1}} ^l \rho_{1,j})j_n ) + w_{2s1} (\rho_{1,2} \alpha j_n + n\eta  j_n) + \ldots + w_{ls1}(\rho_{1,l} \alpha j_n +n \eta j_n)\\
 			w_{1s1}( \rho_{2,1} \alpha  j_n + n\eta  j_n)+ 	w_{2s1}((r_2 \alpha + r_2 \beta + \gamma + n \eta)j_n +(\beta\sum_{\underset{j \neq 2}{j=1}} ^l \rho_{2,j})j_n ) +\ldots +  w_{ls1}(\rho_{2,l} \alpha j_n +n \eta j_n)\\
 			\vdots\\
 				w_{1s1}( \rho_{l,1} \alpha  j_n + n\eta  j_n )+  	w_{2s1}( \rho_{l,2} \alpha  j_n + n\eta j_n )+\ldots+ w_{ls1}((r_l \alpha + r_l \beta + \gamma + n \eta)j_n+  (\beta{\sum_{\underset{j \neq l}{j=1}} ^l} \rho_{l,j}))j_n
 	\end{pmatrix}
 \end{equation*}

\begin{equation*}
	= \frac{1}{\sqrt{n}} \begin{pmatrix}
	(w_{1s1}(r_1 \alpha +\beta({r_1+\sum_{\underset{j \neq 1}{j=1}} ^l \rho_{1,j}})  + \gamma + n \eta ) + w_{2s1} (\rho_{1,2} \alpha  + n\eta ) + \ldots + w_{ls1}(\rho_{1,l} \alpha +n \eta ))j_n\\
(	w_{1s1}( \rho_{2,1} \alpha   + n\eta )+ 	w_{2s1}(r_2 \alpha  +\beta ({r_2 +\sum_{\underset{j \neq 2}{j=1}} ^l \rho_{2,j}})+ \gamma + n \eta ) +\ldots +  w_{ls1}(\rho_{2,l} \alpha  +n \eta))j_n\\
	\vdots\\
	(w_{1s1}( \rho_{l,1} \alpha  + n\eta )+  	w_{2s1}( \rho_{l,2} \alpha   + n\eta )+\ldots+ w_{ls1}(r_l \alpha+  \beta ({r_l+ {\sum_{\underset{j \neq l}{j=1}} ^l} \rho_{l,j}}) +  \gamma + n \eta))j_n
\end{pmatrix}
\end{equation*}

From Equation (\ref{eq4*}), we get
\begin{equation*}
\frac{1}{\sqrt{n}}	\begin{pmatrix}
		\mu_{s,1} w_{1s1}j_n\\
	\mu_{s,1}	w_{2s1}j_n\\
	\vdots\\
	\mu_{s,1}	w_{ls1}j_n
	\end{pmatrix}
= 	\mu_{s,1} \begin{pmatrix}
	\frac{w_{1s1}}{\sqrt{n}} j_n\\
\frac{w_{2s1}}{\sqrt{n}}j_n\\
	\vdots\\
	\frac{w_{ls1}}{\sqrt{n}}j_n
\end{pmatrix}
\end{equation*}
 Hence, for $s = 1,2,\ldots,l$, the value $\mu_{s,1}$ is an eigenvalue of $U(\widehat{G})$ corresponding to the eigenvector $ \frac{1}{\sqrt{n}} \begin{pmatrix}
 	w_{1s1}j_n\\
 	w_{2s1}j_n\\
 	\vdots\\
 	w_{ls1}j_n
 \end{pmatrix}$.\\
Now, 	let $\mu_{s,i}$ be the eigenvalue of $ {C_i
}$   corresponding to the eigenvector $ \widehat{w_{s,i}} = (w_{1si},w_{2si},\ldots,w_{lsi})^T$ for $s= 1,2,\ldots,l$ and $i= 2,3,\ldots,n$. Then, 
 \begin{equation*} 
	{C_i}\widehat{w_{s,i}} =
	\scriptsize
	\begin{pmatrix}
		\lambda_{i,1}\alpha+  (r_1+\sum_{\underset{j \neq 1}{j=1}} ^l \rho_{1,j}) \beta+ \gamma & \rho_{1,2} \alpha& \ldots& \rho_{1,l}\alpha\\
		\rho_{2,1}\alpha& 	\lambda_{i,2}\alpha+  (r_2+\sum_{\underset{j \neq 2}{j=1}} ^l \rho_{2,j}) \beta+ \gamma & \ldots& \rho_{2,l}\alpha\\
		\vdots&\vdots&\ddots&\vdots\\
		
		\rho_{l,1}\alpha & \rho_{l,2}\alpha& \ldots& 	\lambda_{i,l}\alpha+  (r_l+{\sum_{\underset{j \neq l}{j=1}} ^l} \rho_{l,j}) \beta+ \gamma
	\end{pmatrix}
\begin{pmatrix}
	w_{1si}\\
	w_{2si}\\
	\vdots\\
	w_{lsi}
\end{pmatrix}
\end{equation*}

\begin{equation}\label{eq5*}
	=	\begin{pmatrix}
		(\alpha	\lambda_{i,1}+  (r_1+\sum_{\underset{j \neq 1}{j=1}} ^l \rho_{1,j}) \beta+ \gamma )w_{1si}+ \rho_{1,2} \alpha w_{2si}+ \ldots+ \rho_{1,l}\alpha w_{lsi}\\
		\rho_{2,1}\alpha w_{1si}+ (		\lambda_{i,2}\alpha+  (r_2+\sum_{\underset{j \neq 2}{j=1}} ^l \rho_{2,j}) \beta+ \gamma )w_{2si}+ \ldots+ \rho_{2,l}\alpha w_{lsi}\\
		\vdots\\
		\rho_{l,1}\alpha w_{1si}+ \rho_{l,2}\alpha w_{2si}+ \ldots + (	\lambda_{i,l}\alpha+  (r_l+{\sum_{\underset{j \neq l}{j=1}} ^l} \rho_{l,j}) \beta+ \gamma)w_{lsi}
	\end{pmatrix} = 
\mu_{s,i}\begin{pmatrix}
	w_{1si}\\
	w_{2si}\\
	\vdots\\
	w_{lsi}
\end{pmatrix}
\end{equation}
  Now, let $\begin{pmatrix}
  	w_{1si}u_i\\
  	w_{2si}u_i\\
  	\vdots\\
  	w_{lsi}u_i
  \end{pmatrix}$ be a vector of dimension $nl$, we show that this vector is an eigenvector of $U(\widehat{G})$ with respect to the eigenvalue $\mu_{s,i}.$ 
\begin{equation*} 
\begin{pmatrix}
	U(H_1)+ (\beta \sum_{\underset{j \neq 1}{j=1}} ^l \rho_{1,j} ) I & \rho_{1,2} \alpha I + \eta J & \ldots  & \rho_{1,l} \alpha I + \eta J\\
	\rho_{2,1} \alpha I + \eta J&	U(H_2)+ (\beta  \sum_{\underset{j \neq 2}{j=1}} ^l \rho_{2,j} )I& \ldots&  \rho_{2l} \alpha I+ \eta J\\
	\vdots&\vdots& \ddots&\vdots\\
	\rho_{l,1} \alpha I + \eta J&  \rho_{l,2} \alpha I + \eta J& \ldots& U(H_l)+   (\beta {\sum_{\underset{j \neq l}{j=1}} ^l} \rho_{l,j})I\\
\end{pmatrix}
	 \begin{pmatrix}
	w_{1si}u_i\\
	w_{2si}u_i\\
	\vdots\\
	w_{lsi}u_i
\end{pmatrix}
\end{equation*}

\begin{equation*}
	=	\begin{pmatrix}
		w_{1si}(U(H_1)+ (\beta\sum_{\underset{j \neq 1}{j=1}} ^l \rho_{1,j}) I)u_i+  w_{2si}(\rho_{1,2} \alpha I + \eta J )u_i+ \ldots+  w_{lsi}(\rho_{1,l} \alpha I + \eta J)u_i\\
		w_{1si}( \rho_{2,1} \alpha I + \eta J )u_i+w_{2si}(U(H_2)+ (\beta  \sum_{\underset{j \neq 2}{j=1}} ^l \rho_{2,j})I)u_i+ \ldots + 	w_{lsi}( \rho_{2,1} \alpha I + \eta J )u_i\\
		\vdots\\
		w_{1si}( \rho_{l,1} \alpha I + \eta J )u_i+ 	w_{2si}( \rho_{l,2} \alpha I + \eta J )u_i+\ldots+ w_{lsi}(U(H_l)+  (\beta{\sum_{\underset{j \neq l}{j=1}} ^l} \rho_{l,j})I)u_i\\
	\end{pmatrix}
\end{equation*}

\begin{equation*}
	=	\begin{pmatrix}
		w_{1si}(U(H_1)u_i+ (\beta\sum_{\underset{j \neq 1}{j=1}} ^l \rho_{1,j}) Iu_i)+  w_{2si}(\rho_{1,2} \alpha Iu_i + \eta J u_i )+ \ldots+  w_{lsi}(\rho_{1,l} \alpha Iu_i + \eta Ju_i)\\
		w_{1si}( \rho_{2,1} \alpha I u_i + \eta J u_i)+w_{2si}(U(H_2)u_i+ (\beta  \sum_{\underset{j \neq 2}{j=1}} ^l \rho_{2,j})Iu_i)+ \ldots + 	w_{lsi}( \rho_{2,1} \alpha Iu_i + \eta Ju_i )\\
		\vdots\\
		w_{1si}( \rho_{l,1} \alpha I u_i + \eta J u_i)+ 	w_{2si}( \rho_{l,2} \alpha I u_i + \eta Ju_i )+\ldots+ w_{lsi}(U(H_l)u_i+  (\beta{\sum_{\underset{j \neq l}{j=1}} ^l} \rho_{l,j})Iu_i)\\
	\end{pmatrix}
\end{equation*}

\begin{equation*}
	=	\begin{pmatrix}
		w_{1si}((\lambda_{i,1} \alpha + r_1 +\gamma )u_i+ (\beta\sum_{\underset{j \neq 1}{j=1}} ^l \rho_{1,j}) u_i)+  w_{2si}(\rho_{1,2} \alpha u_i +0u_i )+ \ldots+  w_{lsi}(\rho_{1,l} \alpha u_i + 0u_i)\\
		w_{1si}( \rho_{2,1} \alpha  u_i + \eta 0 u_i)+w_{2si}((\lambda_{i,2} \alpha + r_2 +\gamma )u_i+ (\beta  \sum_{\underset{j \neq 2}{j=1}} ^l \rho_{2,j})u_i)+ \ldots + 	w_{lsi}( \rho_{2,1} \alpha u_i + 0u_i )\\
		\vdots\\
		w_{1si}( \rho_{l,1} \alpha  u_i + 0 u_i)+ 	w_{2si}( \rho_{l,2} \alpha  u_i + \eta 0u_i )+\ldots+ w_{lsi}((\lambda_{i,l} \alpha + r_l +\gamma )u_i+  (\beta{\sum_{\underset{j \neq l}{j=1}} ^l} \rho_{l,j})u_i)\\
	\end{pmatrix}
\end{equation*}
 \begin{equation*}
 	=	\begin{pmatrix}
 		((\alpha\lambda_{i,1}+  (r_1+\sum_{\underset{j \neq 1}{j=1}} ^l \rho_{1,j}) \beta+ \gamma )w_{1si}+ \rho_{1,2} \alpha w_{2si}+ \ldots+ \rho_{1,l}\alpha w_{lsi})u_i\\
 		(\rho_{2,1}\alpha w_{1si}+ (\lambda_{i,2}\alpha+  (r_2+\sum_{\underset{j \neq 2}{j=1}} ^l \rho_{2,j}) \beta+ \gamma )w_{2si}+ \ldots+ \rho_{2,l}\alpha w_{lsi})u_i\\
 		\vdots\\
 		(\rho_{l,1}\alpha w_{1si}+ \rho_{l,2}\alpha w_{2si}+ \ldots + (\lambda_{i,l}\alpha+  (r_l+{\sum_{\underset{j \neq l}{j=1}} ^l} \rho_{l,j}) \beta+ \gamma)w_{lsi})u_i
 	\end{pmatrix}
 \end{equation*}
 From equation (\ref{eq5*}), 
 \begin{equation*}
 	\begin{pmatrix}
 	\mu_{s,i}	w_{1si}u_i\\
 	\mu_{s,i}	w_{2si}u_i\\
 		\vdots\\
 	\mu_{s,i}	w_{lsi}u_i
 	\end{pmatrix} = 
 	\mu_{s,i}	\begin{pmatrix}
 	w_{1si}u_i\\
 	w_{2si}u_i\\
 	\vdots\\
	w_{lsi}u_i
 \end{pmatrix}
 \end{equation*}
 Hence, $\mu_{s,i}$ be the eigenvalue of $U(\widehat{G})$ with respect to  eigenvector
 $\begin{pmatrix}
 	w_{1si}u_i\\
 	w_{2si}u_i\\
 	\vdots\\
 	w_{lsi}u_i
 \end{pmatrix}$ for $s= 1,2,\ldots,l$ and $i= 2,3,\ldots,n.$  

\end{proof}

\begin{example}\rm{We consider the $H-$product of graphs given in Figure \ref{fig:path5}.  Let $H = K_{1,3}+ \{12\}$, $H_1 = K_4$, $H_2 = 2K_2$, $H_3 = C_4$, $H_4= K_4$, and the graph   ${\Lambda}_{H} (H_1,H_2,H_3,H_4)$ is shown in Figure \ref{fig:path6}. Let $ (\alpha,\beta,\gamma,\eta)= (2,1,2,1)$, and $U(H_j)$ is universal adjacency matrix of the graph $H_j$ with $4$ vertices. Recall that $Spec_A(H_1)= \{3,-1,-1,-1\}$, $Spec_A(H_2)= \{1,1,-1,-1\}$, $Spec_A(H_3)= \{2,-2,0,0\}$, $Spec_A(H_4)= \{3,-1,-1,-1\}$. With   $ (\alpha,\beta,\gamma,\eta)= (2,1,2,1)$, $Spec( U(H_1))= \{17,5,5,5\}$, $Spec( U(H_2))= \{11,7,3,3\}$, $Spec( U(H_3))= \{15,3,7,7\}$, $Spec( U(H_4))= \{16,4,4,4\}.$ Let 	$ {U}=  \begin{pmatrix}
			1/2 & 1/2& 1/2& 1/2\\
			1/2 & -1/2& -1/2 &1/2\\
			1/2& 1/2& -1/2& -1/2\\
			1/2 & -1/2&1/2 &-1/2
		\end{pmatrix}$ be an orthonormal set of eigenvectors for $A(H_1),A(H_2),A(H_3),$ and $A(H_4)$. Thus, we get the matrices 
		$ \widehat{C_1}=  \begin{pmatrix}
			17 & 6& 6& 4\\
		 	6 & 11& 6 &4\\
		     6& 6& 15& 6\\
		     4 &4 &6 & 16
			\end{pmatrix}$, 	
		$ \widehat{C_2}=  \begin{pmatrix}
			5 & 2& 2& 0\\
			2 & 7& 2 &0\\
			2& 2& 3& 2\\
			0&0 &2 & 4
		\end{pmatrix}$, 
		$ \widehat{C_3}=  \begin{pmatrix}
		5 & 2& 2& 0\\
		2& 3& 2 &0\\
		2& 2& 7& 2\\
		0 &0 &2 & 4
	\end{pmatrix}$, 
				$ \widehat{C_4}=  \begin{pmatrix}
				5 & 2& 2& 0\\
				2& 3& 2 &0\\
				2& 2& 7& 2\\
				0 &0 &2 & 4
			\end{pmatrix}$. And, $Spec(\widehat{C_1})= \{31.0088,12.6673,9.0000,6.3239\}$, $Spec(\widehat{C_2})= \{9.5520,5.0000,$ $3.7840,0.6640\}$, $Spec(\widehat{C_3})= \{9.8399, 5.0000,$ $2.5914,1.5687\}$,  $Spec(\widehat{C_4})= \{9.8399,$ $ 5.0000,2.5914,$ $1.5687\}$. \\
		From Theorem \ref{theorem4.1}, we know that the universal adjacency spectrum of $\widehat{G} = \bigoplus_{H} H_j$ is the eigenvalues of matrices $\widehat{C_1}, \widehat{C_2},\widehat{C_3}, \widehat{C_4}.$ Thus, $Spec(\widehat{G})=  \{31.0088,12.6673,9.8399,9.8399,9.5520,9.0000,$ $6.3239,5.0000,5.0000,3.7840,2.5914,2.5914,1.5687,1.5687,0.6640\}.$
	}
\end{example}

 \begin{corollary}\label{theorem3.2}
Suppose  $H_1,H_2,\ldots,H_l$ are  regular commuting graphs  such that  $H_j$ is an $r_j$ regular graph for $j = 1,2,\ldots,l$. Let $\mu_{i,j}$ be the Laplacian eigenvalues of $H_j,$ where $i = 1,2, \ldots,n,$ and  let $u_1,u_2,\ldots,u_n$ form an orthonormal set of eigenvectors for the Laplacian matrix $L(H_1)$. Then, the Laplacian spectrum of   ${\Lambda}_{H} (H_1,H_2, \ldots,H_l)$ consists of the spectrum of 
	\begin{equation}
		{C_i}=	\begin{pmatrix}
			\mu_{i,1}+(\sum_{\underset{j \neq 1}{j=1}} ^l \rho_{1,j}) & -\rho_{1,2}&\ldots&-\rho_{1,(l-1)}&-\rho_{1,l}\\
			-\rho_{2,1}& \mu_{i,2}+ (\sum_{\underset{j \neq 2}{j=1}} ^l \rho_{2,j})& \ldots&-\rho_{2,(l-1)}&-\rho_{2,l}\\
			\vdots&\vdots&\ddots&\vdots&\vdots\\
			-\rho_{(l-1),1}&-\rho_{(l-1),2}&\ldots&\mu_{i,(l-1)}+(\sum_{\underset{j \neq (l-1)}{j=1}} ^l \rho_{(l-1),j}) &-\rho_{(l-1),l}\\
			-\rho_{l,1}&-\rho_{l,2}&\ldots&-\rho_{l,(l-1)}&\mu_{i,l}+ (\sum_{\underset{j \neq l}{j=1}} ^l \rho_{l,j})
		\end{pmatrix}_{l \times l}
	\end{equation}
	for $i =1,2,\ldots,n$, where 
	\begin{equation*}
		\rho_{j,k} = \rho_{k,j} = \begin{cases}
			1 \text{~if~} jk \in E(H)\\
			0 \text{~otherwise}, 
		\end{cases}
		\text{~for~} 1 \leq j,k \leq l.
	\end{equation*} 
\end{corollary}
\begin{proof}
	By taking $(\alpha,\beta,\gamma,\eta)= (-1,1,0,0)$ in Theorem \ref{theorem4.1},  the result follows directly.
\end{proof}


\begin{corollary}\label{theorem3.3}
Suppose  $H_1,H_2,\ldots,H_l$ are  regular commuting graphs  such that  $H_j$ is an $r_j$ regular graph for $j = 1,2,\ldots,l$. Let $\nu_{i,j}$ be the signless Laplacian eigenvalues of $H_j,$ where $i = 1,2, \ldots,n,$ and  let $u_1,u_2,\ldots,u_n$ form an orthonormal set of eigenvectors for the signless Laplacian matrix $Q(H_1)$. Then, the  signless Laplacian spectrum of  $ {\Lambda}_{H} (H_1,H_2, \ldots,H_l)$ consists of the spectrum of 
	\begin{equation*}
		{C_i}=	\begin{pmatrix}
			\nu_{i,1}+(\sum_{\underset{j \neq 1}{j=1}} ^l \rho_{1,j}) & \rho_{1,2}&\ldots&\rho_{1,(l-1)}&\rho_{1,l}\\
			\rho_{2,1}& \nu_{i,2}+ (\sum_{\underset{j \neq 2}{j=1}} ^l \rho_{2,j})& \ldots&\rho_{2,(l-1)}&\rho_{2,l}\\
			\vdots&\vdots&\ddots&\vdots&\vdots\\
			\rho_{(l-1),1}&\rho_{(l-1),2}&\ldots&\nu_{i,(l-1)}+(\sum_{\underset{j \neq (l-1)}{j=1}} ^l \rho_{(l-1),j}) &\rho_{(l-1),l}\\
			\rho_{l,1}&\rho_{l,2}&\ldots&\rho_{l,(l-1)}&\nu_{i,l}+ (\sum_{\underset{j \neq l}{j=1}} ^l \rho_{l,j})
		\end{pmatrix}_{l \times l}
	\end{equation*}
	for $i =1,2,\ldots,n$, where
	\begin{equation*}
		\rho_{j,k} = \rho_{k,j} = \begin{cases}
			1 \text{~if~} jk \in E(H)\\
			0 \text{~otherwise}, 
		\end{cases}
		\text{~for~} 1 \leq j,k \leq l.
	\end{equation*} 
\end{corollary}
\begin{proof}
		By taking $(\alpha,\beta,\gamma,\eta)= (1,1,0,0)$ in Theorem \ref{theorem4.1},  the result is obtained immediately.
	\end{proof}
	


\begin{corollary}
	Suppose  $H_1,H_2,\ldots,H_l$ are  regular commuting graphs  such that  $H_j$ is an $r_j$ regular graph for $j = 1,2,\ldots,l$. Let $\omega_{i,j}$ be the Seidel matrix eigenvalues of $H_j,$ where $i = 1,2, \ldots,n,$ and  let $u_1,u_2,\ldots,u_n$ form an orthonormal set of eigenvectors for the Seidel matrix $S(H_1)$. Then, the  Seidel spectrum of   $ {\Lambda}_{H} (H_1,H_2, \ldots,H_l)$ consists of the spectrum of 
	 \begin{equation*}
	 	{C_1}=	
	 	\begin{pmatrix}
	 		-2 r_1-1+n & -2\rho_{1,2}+ n& \ldots&-2\rho_{1,(l-1)}+n &  -2\rho_{1,l}+ n\\
	 		-2\rho_{2,1}+n & 		-2 r_2-1+n& \ldots &-2\rho_{2,(l-1)}+n &  -2\rho_{2,l} + n\\
	 		\vdots& \vdots& \ddots& \vdots&\vdots\\
	 		
	 		-2\rho_{(l-1),1}+n& -2\rho_{(l-1),2}+n& \ldots &	-2 r_{(l-1)}-1+n& -2\rho_{(l-1),l}+n\\ 
	 		-2\rho_{l,1} + n & 	-2\rho_{l,2} + n & \ldots & -2\rho_{l,(l-1)}+n &	 	-2 r_l-1+n \\
	 	\end{pmatrix}
	 \end{equation*}
 and
	\begin{equation*}
		{C_i}=	\begin{pmatrix} \scriptsize
		\omega_{i,1}+(\sum_{\underset{j \neq 1}{j=1}} ^l \rho_{1,j}) & -2\rho_{1,2}&\ldots&-2\rho_{1,(l-1)}&-2\rho_{1,l}\\
			-2\rho_{2,1}& \omega_{i,2}+ (\sum_{\underset{j \neq 2}{j=1}} ^l \rho_{2,j})& \ldots&-2\rho_{2,(l-1)}&-2\rho_{2,l}\\
			\vdots&\vdots&\ddots&\vdots&\vdots\\
			-2\rho_{(l-1),1}&-2\rho_{(l-1),2}&\ldots&\omega_{i,(l-1)}+(\sum_{\underset{j \neq (l-1)}{j=1}} ^l \rho_{(l-1),j}) &-2\rho_{(l-1),l}\\
			-2\rho_{l,1}&-2\rho_{l,2}&\ldots&-2\rho_{l,(l-1)}&\omega_{i,l}+ (\sum_{\underset{j \neq l}{j=1}} ^l \rho_{l,j})
		\end{pmatrix}_{l \times l}
	\end{equation*}
	for $i =2,\ldots,n$, where
	\begin{equation*}
		\rho_{j,k} = \rho_{k,j} = \begin{cases}
			1 \text{~if~} jk \in E(H)\\
			0 \text{~otherwise}, 
		\end{cases}
		\text{~for~} 1 \leq j,k \leq l.
	\end{equation*} 
\end{corollary}
\begin{proof}
	By taking $(\alpha,\beta,\gamma,\eta)=  ( -2,0,-1,1)$ in Theorem \ref{theorem4.1}, the result follows immediately.
\end{proof}



\end{document}